\documentclass[12pt,reqno]{amsart}
\usepackage{amsmath,amssymb,amsfonts,amscd,latexsym,amsthm,mathrsfs}
\usepackage{longtable}
\newtheorem{theorem}{Theorem}

\theoremstyle{remark}
\newtheorem{remark}{Remark}
\newtheorem*{question}{Question}

\renewcommand{\d}{{\mathrm d}}

\begin{document}

\title{A generating function of~the squares of~Legendre~polynomials}

\author{Wadim Zudilin}
\address{School of Mathematical and Physical Sciences, The University of Newcastle, Callaghan, NSW 2308, Australia}
\email{wadim.zudilin@newcastle.edu.au}

\thanks{The author is supported by the Australian Research Council.} 

\date{26 November 2012}


\begin{abstract}
We relate a one-parametric generating function for the squares of Legendre polynomials
to an arithmetic hypergeometric series whose parametrisation by a level~7 modular function
was recently given by S.~Cooper. By using this modular parametrisation we resolve
a subfamily of identities involving $1/\pi$ which was experimentally observed by Z.-W.~Sun.
\end{abstract}

\maketitle

In our joint papers \cite{CWZ} with H.\,H.~Chan and J.~Wan and \cite{WZ} with Wan
we made an arithmetic use but also extended the generating functions of Legendre polynomials
$$
P_n(y)={}_2F_1\biggl(\begin{matrix} -n, \, n+1 \\ 1 \end{matrix}\biggm| \frac{1-y}2 \biggr),
$$
originally due to F.~Brafman \cite{Br}. Our generalised generating functions have
the form $\sum_{n=0}^\infty u_nP_n(y)z^n$ where $u_n$ is a so-called
Ap\'ery-like sequence as well as
$$
\sum_{n=0}^\infty{\binom{2n}n}^2P_{2n}(y)z^n \quad\text{and}\quad
\sum_{n=0}^\infty\frac{(3n)!}{n!^3}P_{3n}(y)z^n.
$$
One motivation for the work was a list of formulae for $1/\pi$ given by Z.-W.~Sun~\cite{Sun}. Because
the preprint~\cite{Sun} is a dynamic survey of continuous experimental discoveries by its author,
a few newer examples for $1/\pi$ involving the Legendre polynomials appeared after acceptance of~\cite{CWZ} and~\cite{WZ}.
Namely, the two groups of identities (VI1)--(VI3) and (VII1)--(VII7) related to
the generating functions
\begin{equation}
\sum_{n=0}^\infty P_n(y)^3z^n \quad\text{and}\quad
\sum_{n=0}^\infty\binom{2n}nP_n(y)^2z^n
\label{gen-fun}
\end{equation}
are now given on p.~23 of~\cite{Sun}. A search of existing literature on the subject reveals
no formula which could be useful in proving Sun's observations. The closest-to-wanted identity
is Bailey's
\begin{align}
\sum_{n=0}^\infty P_n(x)P_n(y)z^n
&=\frac1{\bigl(1+z(z-2\sqrt{(1-x^2)(1-y^2)}-2xy)\bigr)^{1/2}}
\nonumber\\ &\qquad\times
{}_2F_1\biggl(\begin{matrix} \frac12, \, \frac12 \\ 1 \end{matrix}\biggm|
\frac{-4\sqrt{(1-x^2)(1-y^2)}z}{1+z(z-2\sqrt{(1-x^2)(1-y^2)}-2xy)}\biggr),
\label{Bailey}
\end{align}
which follows from \cite[Eqs.~(2.1) and (3.1)]{Ba2} and \cite[Eq.~(7) on p.~81]{Ba1}.
The generating function \eqref{Bailey} was rediscovered later by Maximon~\cite{Ma}. It admits,
in fact, a less radical form
\begin{equation}
\sum_{n=0}^\infty P_n(x)P_n(y)z^n
=\frac1{(1-2xyz+z^2)^{1/2}}
\,{}_2F_1\biggl(\begin{matrix} \frac14, \, \frac34 \\ 1 \end{matrix}\biggm|
\frac{4(1-x^2)(1-y^2)z^2}{(1-2xyz+z^2)^2}\biggr)
\label{Wan}
\end{equation}
which is due to Wan~\cite{Wa}.
Unfortunately, no simple generalisation of the result for the terms on the left-hand
side twisted by the central binomial coefficients is known, even in the particular case $x=y$.

With the help of Clausen's identity
$$
P_n(y)^2
={}_3F_2\biggl(\begin{matrix} -n, \, n+1, \, \frac12 \\ 1, \, 1 \end{matrix}\biggm| 1-y^2 \biggr)
=\sum_{k=0}^n\binom nk\binom{n+k}n\binom{2k}k\biggl(-\frac{1-y^2}4\biggr)^k,
$$
we find that the second generating function in~\eqref{gen-fun} is equivalent to
\begin{equation}
\sum_{n=0}^\infty\binom{2n}nz^n\sum_{k=0}^n\binom nk\binom{n+k}n\binom{2k}kx^k.
\label{gen1}
\end{equation}
In view of \cite[Theorem~1]{WZ}, its Clausen-type
specialization~\cite{CTYZ} and our identities \eqref{main1}, \eqref{main2} below,
it is quite likely that the latter generating function can be written
as a product of two arithmetic hypergeometric series, each satisfying a second order
linear differential equation. In this note we only recover the special case
$z=x/(1+x)^2$ of the expected identity, the case which is suggested by
Sun's observations (VII1) and (VII3)--(VII6) from~\cite{Sun}.

\begin{theorem}
\label{th1}
For $v$ from a small neighbourhood of the origin, take
$$
x(v)=\frac{v}{1+5v+8v^2}
\quad\text{and}\quad
z(v)=\frac{x(v)}{(1+x(v))^2}=\frac{v(1+5v+8v^2)}{(1+2v)^2(1+4v)^2}.
$$
Then
\begin{equation}
\sum_{n=0}^\infty\binom{2n}nz(v)^n\sum_{k=0}^n\binom nk\binom{n+k}n\binom{2k}kx(v)^k
=\frac{1+2v}{1+4v}\sum_{n=0}^\infty u_n\biggl(\frac v{(1+4v)^3}\biggr)^n,
\label{main1}
\end{equation}
where the sequence \textup{\cite[A183204]{OEIS}}
\begin{equation*}
u_n=\sum_{k=0}^n{\binom nk}^2\binom{n+k}n\binom{2k}n
=\sum_{k=0}^n(-1)^{n-k}\binom{3n+1}{n-k}{\binom{n+k}{n}}^3
\end{equation*}
satisfies the Ap\'ery-like recurrence equation
\begin{equation*}
\begin{gathered}
(n+1)^3u_{n+1}=(2n+1)(13n^2+13n+4)u_n+3n(3n-1)(3n+1)u_{n-1}
\\
\quad\text{for}\; n=0,1,2,\dots,
\quad u_{-1}=0, \; u_0=1.
\end{gathered}
\end{equation*}
\end{theorem}

Because $y^2=1+4x$ for the second generating function in~\eqref{gen-fun},
the equivalent form of~\eqref{main1} is the identity
\begin{multline*}
\sum_{n=0}^\infty\binom{2n}nP_n\biggl(\frac{\sqrt{(1+v)(1+8v)}}{\sqrt{1+5v+8v^2}}\biggr)^2\biggl(\frac{v(1+5v+8v^2)}{(1+2v)^2(1+4v)^2}\biggr)^n
\\
=\frac{1+2v}{1+4v}\sum_{n=0}^\infty u_n\biggl(\frac v{(1+4v)^3}\biggr)^n,
\end{multline*}

\begin{remark}
\label{rem1}
S.~Cooper constructs in~\cite[Theorem~3.1]{Co} a modular parametrisation of the generating
function $\sum_{n=0}^\infty u_nw^n$. Namely, he proves that the substitution
\begin{equation}
w(\tau)=\frac{\eta(\tau)^4\eta(7\tau)^4}{\eta(\tau)^8+13\eta(\tau)^4\eta(7\tau)^4+49\eta(7\tau)^8}
\label{mod-par}
\end{equation}
translates the function into the Eisenstein series
$(7E_2(7\tau)-E_2(\tau))/6$. Here $\eta(\tau)=q^{1/24}\prod_{m=1}^\infty(1-q^m)$
is Dedekind's eta function, $q=e^{2\pi i\tau}$, and
$$
E_2(\tau)=\frac{12}{\pi i}\,\frac{\d\log\eta}{\d\tau}=1-24\sum_{n=1}^\infty\frac{q^n}{1-q^n}.
$$
Using this, Cooper derives a general family \cite[Eqs.~(37),~(39)]{Co}
of the related Ramanu\-jan-type identities for $1/\pi$. It is this result
and the `translation' method \cite{Zu} which allow us to prove Sun's observations
(VII1) and (VII3)--(VII6) from~\cite{Sun}.
Note that this modular parametrisation and results of Chan and Cooper \cite[Lemmas 4.1 and 4.3]{CC1}
lead to the following hypergeometric forms of the generating function:
\begin{align*}
&
\frac1{\sqrt{1+13h+49h^2}}\sum_{n=0}^\infty u_n\biggl(\frac h{1+13h+49h^2}\biggr)^n
\\ &\quad
=\frac1{\sqrt{1+245h+2401h^2}}
\,{}_3F_2\biggl(\begin{matrix} \frac16, \, \frac12, \, \frac56 \\ 1, \, 1 \end{matrix}\biggm|
\frac{1728h}{(1+13h+49h^2)(1+245h+2401h^2)^3}\biggr)
\\ &\quad
=\frac1{\sqrt{1+5h+h^2}}
\,{}_3F_2\biggl(\begin{matrix} \frac16, \, \frac12, \, \frac56 \\ 1, \, 1 \end{matrix}\biggm|
\frac{1728h^7}{(1+13h+49h^2)(1+5h+h^2)^3}\biggr)
\end{align*}
which are valid near $h=0$.
\end{remark}

\begin{theorem}[Satellite identity]
\label{th2}
The identity
\begin{align}
&
\sum_{n=0}^\infty\binom{2n}n\biggl(\frac x{(1+x)^2}\biggr)^n\sum_{k=0}^n\binom nk\binom{n+k}n\binom{2k}kx^k
\nonumber\\ &\qquad\times
\bigl(2x(3+4x)-n(1-x)(3+5x)+4k(1+x)(1+4x)\bigr)
=0
\label{main2}
\end{align}
is valid whenever the left-hand side makes sense.
\end{theorem}

\begin{proof}[Proof of Theorems~\textup{\ref{th1}} and~\textup{\ref{th2}}]
The identity~\eqref{main1} is equivalent to
\begin{multline*}
\sum_{n=0}^\infty\binom{2n}n\frac{v^n(1+5v+8v^2)^n}{(1+2v)^{2n+1}(1+4v)^{2n+1}}
\sum_{k=0}^n\binom nk\binom{n+k}n\binom{2k}k\frac{v^k}{(1+5v+8v^2)^k}
\\
=\sum_{n=0}^\infty u_n\frac{v^n}{(1+4v)^{3n+2}}.
\end{multline*}
It is routine to verify that the both sides are annihilated by the differential operator
\begin{align*}
&
v^2(1+v)(1+8v)(1+5v+8v^2)\frac{\d^3}{\d v^3}
+3v(1+21v+122v^2+280v^3+192v^4)\frac{\d^2}{\d v^2}
\\ &\qquad
+(1+50v+454v^2+1408v^3+1216v^4)\frac{\d}{\d v}
+4(1+22v+108v^2+128v^3),
\end{align*}
and the proof of Theorem~\ref{th1} follows. A similar routine shows the vanishing
in Theorem~\ref{th2}.
\end{proof}

\begin{table}[h]
\begin{center}
\begin{tabular}{|c|c|c|c|c|c|p{6 in}|}
\hline
\# in~\cite{Sun} & $x$ & $z$ & $v$ & $w=v/(1+4v)^3$ & $\tau \vphantom{\big|^1}$ \\
\hline
(VII1) & $-\frac1{14}$ & $\frac{14}{225}$ & $1$ & $\frac1{5^3}$ & $\frac{2i}{\sqrt7} \vphantom{\big|^1}$ \\[2mm]
(VII2) & $\frac9{20}$ & $-\frac5{196}$ & & & \\[2mm]
(VII3) & $-\frac1{21}$ & $\frac{21}{484}$ & $1+\frac{\sqrt{14}}4$ & $\frac{188-42\sqrt{14}}{22^3}$ & $\frac{i\sqrt{6}}{\sqrt7}$ \\[2mm]
(VII4) & $-\frac1{45}$ & $\frac{45}{2116}$ & $\frac52+\frac{7\sqrt2}4$ & $\bigl(\frac{8-3\sqrt2}{46}\bigr)^3$ & $\frac{i\sqrt{10}}{\sqrt7}$ \\[2mm]
(VII5) & $\frac17$ & $-\frac7{36}$ & $-\frac34-\frac{\sqrt7}4$ & $\frac{-34+14\sqrt7}{6^3}$ & $\frac{i\sqrt{3}}{\sqrt7}$ \\[2mm]
(VII6) & $\frac1{175}$ & $-\frac{175}{30276}$ & $-\frac{45}4-\frac{17\sqrt7}4$ & $\bigl(\frac{-13+7\sqrt7}{174}\bigr)^3$ & $\frac{i\sqrt{19}}{\sqrt7}$ \\[2mm]
(VII7) & $-\frac{576}{3025}$ & $\frac{3025}{188356}$ & & & \\[2mm]
\hline
\end{tabular}
\end{center}
\bigskip
\caption{The choice of parameters for Sun's observations in \cite[p.~23]{Sun}.
The last column corresponds to the choice of $\tau$ such that $w(\tau)=v/(1+4v)^3$
for the modular function $w(\tau)$ defined in~\eqref{mod-par}}
\label{table1}
\end{table}

In Table~\ref{table1} we list the relevant parametrisations of Sun's formulae from~\cite{Sun}.
The last column corresponds to the choice of $\tau$ in~\eqref{mod-par} such that
$v/(1+4v)^3=w(\tau)$ there. The general formulae for $1/\pi$ in these cases,
\begin{equation}
\sum_{n=0}^\infty(a+bn)u_nw^n=\frac1{\pi\sqrt7},
\label{pi-given}
\end{equation}
are given by Cooper in~\cite[Eq.~(37)]{Co}. On using \eqref{main1} and its $v$-derivative
\begin{multline*}
\begin{aligned}
&
\sum_{n=0}^\infty\binom{2n}nz(v)^n\sum_{k=0}^n\binom nk\binom{n+k}n\binom{2k}kx(v)^k
\\ &\qquad\times
\biggl(n\,\frac{(1-8v^2)(1+4v+8v^2)}{v(1+2v)(1+4v)(1+5v+8v^2)}+k\,\frac{1-8v^2}{v(1+5v+8v^2)}\biggr)
\end{aligned}
\\
=\frac{1+2v}{1+4v}\sum_{n=0}^\infty u_n\frac{v^n}{(1+4v)^{3n}}
\biggl(n\,\frac{1-8v}{v(1+4v)}-\frac2{(1+2v)(1+4v)}\biggr),
\end{multline*}
the equalities~\eqref{pi-given} together with the related specialisations of~\eqref{main2}
(to eliminate the linear term in~$k$)
imply Sun's identities (VII1), (VII3)--(VII6) by translation~\cite{Zu}.

\medskip
Note that Cooper's \cite[Table~1]{Co} involves two more examples corresponding to the choices
$-1/4^3$ and $-1/22^3$ for $v/(1+4v)^3$; the values of $x$ and $z$ in these cases are
zeroes of certain irreducible cubic polynomials though. There are also several examples
when $x$ and $z$ are taken from a quadratic field. For instance, taking $\tau=\frac{i\sqrt{11}}{\sqrt7}$
one gets
$$
x=\frac{23-8\sqrt{11}}{175} \quad\text{and}\quad z=\frac{83-32\sqrt{11}}{1100}
$$
in~\eqref{main1} and~\eqref{main2}; the corresponding $v_1=-6.798\dots$ and $v_2=-0.018\dots$ solve
the quartic equation $64v^4+448v^3+96v^2+56v+1=0$. As such identities are only of
theoretical importance, we do not derive them here.

\medskip
It is apparent that there is a variety of formulae similar to~\eqref{main1}
and~\eqref{main2} designed for generating functions of other polynomials. For example, Sun's list
contains five identities involving values of the polynomials
$$
A_n(x)=\sum_{k=0}^n{\binom nk}^2\binom{n+k}nx^k, \quad n=0,1,2,\dotsc.
$$
By examining the entries (2.1)--(2.3) on p.~3 of~\cite{Sun} one notices that
the parameters $x$ and $z$ of the generating function
\begin{equation}
\sum_{n=0}^\infty\binom{2n}nA_n(x)z^n
\label{gen2}
\end{equation}
are related by $z=x/(1-4x)$, while the entries (6.1) and (6.2) on p.~15 there
correspond to the relation $z=1/(x+1)^2$. With some work we find that those
specialisations indeed lead to third order arithmetic linear differential equations
which can be then identified with the known examples \cite{CC2,CZ}:
\begin{multline*}
\sum_{n=0}^\infty\binom{2n}n\frac{v^n(1-v)^n(1-4v)^n}{(1-2v+4v^2)^{2n+1}}
\sum_{k=0}^n{\binom nk}^2\binom{n+k}n\frac{v^k(1-v)^k(1-4v)^k}{(1-4v^2)^{2k+1}}
\\
\begin{aligned}
&=\sum_{n=0}^\infty\sum_{k=0}^n{\binom nk}^2{\binom{n+k}n}^2
\frac{v^n(1-2v)^n(1-4v)^{2n}}{(1-v)^{n+1}(1+2v)^{n+1}}
\\
&=\sum_{n=0}^\infty\sum_{k=0}^n{\binom nk}^2\binom{2k}k\binom{2n-2k}{n-k}
\frac{(-1)^nv^n(1-v)^n(1-4v^2)^n}{(1-4v)^{2n+2}}
\\
&=\sum_{n=0}^\infty\frac{(3n)!}{n!^3}\binom{2n}n
\frac{v^n(1-v)^n(1-4v^2)^n(1-4v)^{4n}}{(1+4v-8v^2)^{2n+2}}
\end{aligned}
\end{multline*}
and
\begin{multline*}
\sum_{n=0}^\infty\binom{2n}n\frac{v^{2n}}{(1+10v+27v^2)^{2n+1}}
\sum_{k=0}^n{\binom nk}^2\binom{n+k}n\frac{(1+9v+27v^2)^k}{v^k}
\\
=\sum_{n=0}^\infty\frac{(3n)!}{n!^3}\binom{2n}n\frac{v^n(1+9v+27v^2)^n}{(1+9v)^{6n+2}},
\end{multline*}
respectively. Additionally, there are satellite identities for each of
the specialisations, both similar to~\eqref{main2}. These identities,
the known Ramanujan-type formulae for the right-hand sides and the translation
technique can be then used to prove Sun's observations.

On the other hand, as already mentioned at the beginning, it is natural
to expect the existence of Bailey--Brafman-like identities~\cite{CWZ,WZ}
for the two-variate generating functions \eqref{gen1},~\eqref{gen2}.

\begin{question}
Given an (arithmetic) generating function $\sum_{n=0}^\infty A_nz^n$
which satisfies a second order linear differential equation (with regular
singularities), is it true that $\sum_{n=0}^\infty\binom{2n}nA_nz^n$
can be written as the product of two arithmetic series, each satisfying
(its own) second order linear differential equation?
\end{question}

Here, of course, we allow $A_n$ depend on some other parameters;
the example of such a product decomposition for $A_n=A_n(x)=\sum_k{\binom nk}^2\binom{2k}nx^n$
is given recently by M.~Rogers and A.~Straub \cite[Theorem~2.3]{RS}.
An affirmative answer to the question would give one an arithmetic
parametrisation of the generating function $\sum_{n=0}^\infty\binom{2n}nP_n(x)P_n(y)z^n$
(cf.~\eqref{Bailey} or~\eqref{Wan}).

Note that there are some other generating functions in~\cite{Sun},
like the first one in~\eqref{gen-fun}, which are not of the form
$\sum_{n=0}^\infty\binom{2n}nA_nz^n$. We believe however that they
can be reduced to the latter form by a suitable algebraic transformation.

\medskip
\textbf{Acknowledgments.} I would like to thank Heng Huat Chan, Jes\'us Guillera and James Wan for our fruitful conversations
on the subject, and Shaun Cooper for useful comments and for making me familiar with his work~\cite{Co} already
at the time of its preparation.

\end{document}